\documentclass[12pt]{amsart}
\usepackage{amsmath,amssymb,amsthm}
\usepackage{tikz-cd} 
\usepackage[utf8]{inputenc}
\DeclareMathOperator{\Pic}{\mathrm{Pic}}

\begin{document}

\title[]
{\textbf{ Strictly nef bundles }}
\author{Priyankur Chaudhuri}
\address{Department of Mathematics, University of Maryland, College Park, MD 20742}
\email{pchaudhu@umd.edu}

\begin{abstract}
In this short note we will show that every homogeneous strictly nef vector bundle on a complex flag variety is ample. We also consider the question of when ampleness of vector bundles can be tested on curves.
\end{abstract}

\maketitle\section{Introduction} 

A line bundle $L$ on a projective variety $X$ is called \emph{strictly nef} if $\deg(L|_{C}) > 0$ for all integral curves $C \subset X$. A vector bundle $ \mathcal{E}$ on X is called strictly nef if the corresponding hyperplane bundle  $\mathcal{O}_{\mathbb{P} \mathcal{E}}(1)$ is strictly nef on  $\mathbb{P} \mathcal{E}$, where $\mathbb{P}\mathcal{E}$ is the projective bundle over $X$ associated to $\mathcal{E}$. There are old examples of Mumford and Ramanujam (example 10.6, page 56 and example 10.8, page 57 in \cite{Hartshorne1}) which show that a strictly nef (even effective) line bundle needn't be ample. However, it is known that strictly nef line bundles on homogeneous varieties are ample. For abelian varieties, this is Prop 1.4 in \cite{Serrano}, whereas for flag varieties (being Fano), it follows from the basepoint free theorem (Theorem 3.3 in \cite{KM}): if $L$ is a strictly nef line bundle on a smooth Fano variety $X$, then $L-K_X$ is ample, hence $L$ is semiample. Now, it is easy to see that $L$ must be ample: replacing $L$ by a positive tensor power, we may assume that $L$ is globally generated. Let $\phi: X \rightarrow Y$ be the morphism given by the linear system $|L|$. Then $L=\phi^*(A)$ for some ample $A \in \Pic Y$. If $\phi(C)=pt$ for some curve $C \in X$, then $L|_C=\phi^*(A)|_C =\mathcal{O}_C$ which can't be since $L$ is strictly nef. Thus $\phi$ is finite and $L$, being a finite pullback of an ample divisor is also ample.\\

Given all this, it is natural to wonder what happens for strictly nef bundles of higher rank on homogeneous varieties. 
We will see in Section 1 that strictly nef bundles on a flag variety which are homogeneous, are indeed ample. In Section 2, we will consider question 2.5 in \cite{BHN} which asks whether for vector bundles on an abelian variety, ampleness on curves implies ampleness.
\section {Homogeneous bundles on flag varieties}

	Let $G$ be a complex semisimple Lie group, $B$ be a fixed Borel subgroup containing a maximal torus $T$. Let $R$ be the set of roots of $G$ with respect to $T$ and let $R_-$ denote the set of roots in the Lie algebra of $B$, the so-called negative roots. Let $P$ be a parabolic subgroup defined by a subset $I$ of the set of simple positive roots $ \{\alpha_1,....,\alpha_l \}$. For any root $\alpha$ of $\mathfrak{g}$ with respect to $\mathfrak{t}$, we can choose a representative $X_{\alpha}$ of the one-dimensional root space $\mathfrak{g}_{\alpha}$ such that if $Z_{\alpha}:=[X_{\alpha},X_{-\alpha}]$, then $\alpha(Z_{\alpha})=2$.  Let ${\lambda_{1},..., \lambda_{l}}$ be the \emph{fundamental weights}, i.e., such that $\langle\lambda_{i}, \alpha_{j}\rangle=\delta_{ij}$.     Here $ \langle, \rangle$ denotes the canonical pairing between roots and weights ( see page 6 of \cite {Snow}). Then any arbitrary weight $\lambda$ can be written as $\lambda = n_{1}\lambda_{1}+...+n_{l}\lambda_{l}$ where $n_i \in \mathbb{Z}$ $\forall i. $  If $n_{i}\geq 0$ $\forall i,$ then we say $\lambda$ is \emph{dominant}. Let $\Lambda^{+}$ denote the set of dominant weights. There exists a partial order on the weights: $\mu < \lambda$ if $\lambda - \mu$ is a positive (possibly 0) linear combination of dominant weights. Weights maximal with respect to this partial order will be called \emph{maximal} weights. \\

Let $\mathcal{E}= G\times^{P}E_{0}$ be a homegeneous vector bundle on $X$=$G/P$ given by a $P$ module $E_{0}$. In what follows, $\Lambda(E_{0})$ will denote the weights of $E_{0}$ and $\Lambda _{\mathrm{max}}(E_{0})$ will denote the subset of maximal weights. Here are a few facts we will be using:\\

(A) Corollary 11.5 on page 47 of \cite{Snow}  says: Let $\mathcal{E}= G\times^{P}E_{0}$ be a homogeneous vector bundle on X = $G/P$ where $P$ is a parabolic subgroup of $G$. If $\Lambda_{\mathrm{max}}(E_{0})\subset\Lambda^{+}$, then $  \mathcal{E}$ is generated by its global sections.\\

(B) We will also need the following well known fact: If $L(\lambda)\in \Pic G/P$ is the homogeneous line bundle corresponding to the weight $\lambda$ (see beginning of page 18 of \cite{Snow} for details on the correspondence) and if $C(\alpha)\cong\mathbb{P}^{1}$ is the rational curve in $G/P$ corresponding to the roots $\pm\alpha$ (see the proof of Prop 4.4 in page 231 of \cite{Jant} for details. As a flag variety, $C(\alpha)= S_{\alpha}/B_{\alpha}$, where $\mathfrak{s}_{\alpha}:= \mathbb{C}Z_{\alpha}+\mathfrak{g}_{\alpha}+\mathfrak{g}_{-\alpha}$ and $\mathfrak{b}_{\alpha}:=\mathbb{C}Z_{\alpha}+\mathfrak{g}_{-\alpha}$.), then $\deg L(\lambda)|_{C(\alpha)}= \langle \lambda,\alpha \rangle$. 

\newtheorem{thm}{Theorem} 
\begin{thm}
 If a homogeneous vector bundle $\mathcal{E}$ on $G/P$ is strictly nef, then it is globally generated and hence ample.
 \end{thm}
 \begin{proof}
 Let $\mathcal{E}= G\times^{P}E_{0}$ be given by the $P$ module $E_{0}$ and $P$ correspond to the index set $I$. In what follows, we say $j\notin I $ to mean $\alpha_j\notin I$. Let $ \Lambda(E_0)= \{\Lambda_1,....,\Lambda_r\}$ counted with multiplicities, where $ \Lambda_k = \sum_{i=1}^l n_{i(k)}\lambda_i $. The main idea is to understand the splitting of $\mathcal{E}|_{C(\alpha_j)} \forall j \notin I$. If $\mathfrak{t}$ denotes the Lie algebra of the maximal torus in $X$ and $\mathfrak{t}_{C(\alpha_j)}$ that of the torus in $C_{(\alpha_j)} \forall j \notin I$, then $\mathfrak{t}_{C(\alpha_j)}= \mathbb{C}(Z_{\alpha_j})$. Since $\forall k=1,...,l$, the fundamental weights $\lambda_k$ are linear functionals $\lambda_k: \mathfrak{t} \rightarrow  \mathbb{C}$, the restriction $\lambda_k|_{\mathfrak{t}_{C(\alpha_j)}}$ is determined by $\lambda_k(Z_{\alpha_j})=\langle\lambda_{k}, \alpha_{j}\rangle =\delta_{kj}$. Thus $\Lambda(E_0|_{C(\alpha_j)})$ = \{$n_{j(1)}\lambda_j,....., n_{j(r)}\lambda_j  \}$. Suppose $\mathcal{E}|_{C_{(\alpha_j)}}$ gives a representation $\mathbb{C}Z_{\alpha_j} \rightarrow M_r(\mathbb{C})$ where $ r$ is the rank of $\mathcal{E}$. $\mathbb{C}Z_{\alpha_j}$ being one dimensional, this splits as a direct sum of one dimensional representations given by these weights. Thus $\mathcal{E}|_{C(\alpha_j)}= \mathcal{L}(n_{j(1)}\lambda_j)|_{C(\alpha_j)} \oplus ... \oplus \mathcal{L}(n_{j(r)}\lambda_j)|_{C(\alpha_j)}$, where the direct summands denote the restrictions of the line bundles on $X$ corresponding to the indicated weights. This is $\cong$ $ \mathcal{O}(n_{j(1)}) \oplus....\oplus \mathcal{O}(n_{j(r)})$ $\forall j \notin I $ by (B) above. But since all line bundle quotients of a strictly nef bundle restricted to a curve have positive degree (see Prop 2.1 in \cite {LOY}), thus $n_{j(k)}>0$ $\forall j \notin I$, $\forall k = 1,...,r$. Now consider the projection $G/B \xrightarrow{\pi} G/P$. Since $\pi(C(\alpha_j)) = pt$ $ \forall j\in I$, $(\pi^*\mathcal{E})|_{C(\alpha_j)}$ is trivial $\forall j \in I$ which means that $n_{j(k)}=0$ $ \forall j \in I$, $ \forall k= 1,...,r$ arguing as above. Thus $n_{j(k)} \geq 0$ $\forall j$ and for all weights $\Lambda_k$, $k=1,...,r$ and thus $ \mathcal{E} $ is spanned by (A) above. Now by the canonical surjection $\pi^*\mathcal{E} \xrightarrow{}  \mathcal{O}_{\mathbb{P}\mathcal{E}}(1)$,  $\mathcal{O}_{\mathbb{P}\mathcal{E}}(1)$ is generated by its global sections  and strictly nef, hence ample.
 
 \end{proof}

 \textbf{Remark 1:} In view of Proposition 6.1.8 (ii) of \cite{Laz 2}, the arguments of the above proof show the more general fact that homogeneous bundles on $G/P$ which are nef on the finite set of curves $C(\alpha_j)$ for $j \notin I$ are globally generated. 
 \\
 
 \textbf{Remark 2:} The above theorem has also been proved independently around the same time by \cite{BHN}. (See Theorem 3.1 in \cite{BHN}.)
 
 \section {Ampleness for bundles on curves}
  In this section, we consider question 2.5 of \cite{BHN} which asks (in analogy with the line bundle case): Is a vector bundle on an abelian variety ample if its restriction to every curve is ample? 
 We have the following lemma whose proof is an easy application of Hartshorne's ampleness criterion for vector bundles on a smooth projective curve (Theorem 2.4 in \cite{Hartshorne 2})\\
 
\newtheorem{lem}{Lemma}
\begin{lem}
 Let $ \mathcal{E}$ be a vector bundle on a projective variety $X$. Then $\mathcal{E}$ is ample when restricted to curves iff $ \forall $ finite morphisms $f: C \rightarrow   X$ where $C$ is a smooth projective curve, all vector bundle quotients of $f^*(\mathcal{E})$ are of positive degree.
 \end{lem}
\begin{proof}
$ \impliedby$: If $C\subset X$ is a curve, consider its normalization $\Tilde{C} \xrightarrow{f} C \subset X$. Then $f^*(\mathcal{E})$ is ample on $\Tilde{C}$ by Hartshorne's criterion. Hence $\mathcal{E}|_C$ is also ample by Prop 6.1.8 (iii) in \cite {Laz 2}.
 \\ $\implies$: Let $f: C \rightarrow X $  be a finite morphism from a smooth projective curve and let $C^{'}= f(C)$. Then $\mathcal{E}|_{C^{'}}$ is ample, thus $f^*(\mathcal{E}|_{C^{'}})$ is ample on $C$. By Hartshorne's ampleness criterion on $C$, we are done.
 \end{proof}
 \newtheorem{cor}{Corollary}
 \begin{cor}Let $ X \xrightarrow{\pi}  Y $ be a finite morphism of projective varieties, $\mathcal{E}$ be a bundle on $Y$ that is ample when restricted to curves. Then $\pi^*(\mathcal{E})$ is also ample when restricted to curves.
 \end{cor}
 \begin{proof}
 Let $C\xrightarrow{f}   X $ be a finite morphism, where $C$ is a smooth projective curve. Let $ f ^* \pi ^*(\mathcal{E}) \xrightarrow{}   Q \rightarrow   0 $ be a quotient bundle. Now $ \pi \circ f : C \rightarrow   Y$ is finite and $\deg(Q) > 0$ by ampleness of $ f^* \circ \pi^*(\mathcal{E})$. Thus $ \pi^*(\mathcal{E})$ is  ample by above lemma.
 \end{proof} 
 \textbf{Remark 3:} Let $X$ be an $n$-dimensional projective variety. Then $X$ admits a finite surjective morphism, say $f: X \rightarrow \mathbb{P}^n$. If $E$ is a vector bundle on $\mathbb{P}^n$ which is ample on curves, so is $f^*(E)$ by the above corollary. Moreover, by Prop 6.1.8 (iii) in \cite {Laz 2}, $E$ is ample iff $f^*(E)$ is. Thus, if there exists a vector bundle on $\mathbb{P}^n$ for some $n$, which is ample on curves, but not ample, that would negatively answer the question of \cite {BHN}.
\\

\textbf{Acknowledgements:}
  I thank my advisor Patrick Brosnan for suggesting the problem on homogeneous bundles and many helpful conversations and suggestions which improved this paper. Thanks to Krishna Hanumanthu for alerting me about an error in an earlier version of this paper, where I was assuming the bundle in the proof of Prop 5 in \cite{Fulton} to be ample on curves. Thanks also to J\'anos Koll\'ar for a fruitful conversation and to the referee whose comments helped to improve the exposition.
 

\end{document}